\documentclass[a4paper]{article}

\usepackage[utf8]{inputenc}
\usepackage{lmodern}
\usepackage{amssymb,amsfonts,amsthm,amsmath,bbm,mathrsfs,aliascnt,mathtools}
\usepackage[english]{babel}
\usepackage[colorlinks]{hyperref}
\usepackage{tikz}
\usepackage{verbatim}

\theoremstyle{plain}
\newtheorem{theorem}{Theorem}[section]
\newtheorem{corollary}[theorem]{Corollary}
\newtheorem{lemma}[theorem]{Lemma}
\newtheorem{proposition}[theorem]{Proposition}

\theoremstyle{definition}
\newtheorem{definition}[theorem]{Definition}

\theoremstyle{remark}
\newtheorem{remark}[theorem]{Remark}

\numberwithin{equation}{section}

\newcommand{\N}{\mathbb{N}}
\newcommand{\R}{\mathbb{R}}

\newcommand{\Y}{\mathbf{Y}}

\renewcommand{\S}{\mathbf{S}}
\newcommand{\NN}{\mathbf{N}}
\newcommand{\Z}{\mathbf{Z}}

\newcommand{\ind}[1]{\mathbf{1}_{\left\{#1\right\}}}

\newcommand{\crochet}[1]{{\langle #1 \rangle}}
\renewcommand{\bar}[1]{\overline{#1}}
\renewcommand{\tilde}[1]{\widetilde{#1}}

\newcommand{\e}{\mathrm{e}}
\newcommand{\dd}{\mathrm{d}}
\newcommand{\egaldistr}{{\overset{(d)}{=}}}

\DeclareMathOperator{\E}{\mathbb{E}}

\renewcommand{\P}{\mathbb{P}}

\newcommand{\calP}{\mathcal{P}}

\newcommand{\x}{\mathbf{x}}
\newcommand{\y}{\mathbf{y}}
\newcommand{\X}{\mathbf{X}}

\renewcommand{\rho}{\varrho}
\renewcommand{\epsilon}{\varepsilon}

\title{Branching-stable point measures
and processes}
\author{Jean Bertoin\footnote{Institute of Mathematics, University of Zurich, Switzerland} \and Aser Cortines\footnote{Institute of Mathematics, University of Zurich, Switzerland} \and Bastien Mallein\footnote{LAGA - Institut Galil\'ee, Universit\'e Paris 13, France}}
\date{\today}

\begin{document}

\maketitle

\begin{abstract}
We introduce and study the class of branching-stable point measures, which can be seen as an analog of stable random variables when the branching mechanism for point measures replaces the usual addition. In contrast with the classical theory of stable (L\'evy) processes, there exists a rich family of branching-stable point measures with \emph{negative} scaling exponent, which can be described as certain Crump-Mode-Jagers branching processes. We investigate the asymptotic behavior of their cumulative distribution functions, that is, the number of atoms in $(-\infty, x]$ as $x\to \infty$, and further depict the genealogical lineage of typical atoms. For both results, we rely crucially on the work of Biggins.
\end{abstract}

\medskip


\section{Introduction}
\label{sec:introduction}
Recall that a random variable $\xi_1$ with values in $\R^d$ has a (strictly) stable distribution if and only if for every integer $n\geq 2$, the sum of $n$ i.i.d. copies of $\xi_1$ has the same law as $a(n)\xi_1$, where $a(n)$ is some sequence of positive numbers. Excluding implicitly the degenerate case when $\xi_1\equiv 0$, there exists then an index $\alpha\in(0,2]$ such that $a(n)=n^{1/\alpha}$, and the distribution of $\xi_1$ is called $\alpha$-stable. This family of probability laws arises naturally in a variety of weak limit theorems and plays therefore an important role in the analysis of many random processes (see, for instance, \cite{Bogetal, GeWe, Sato, SamTa} and references therein). In this article, we introduce and study a class of random point measures satisfying an analogous property, where the addition of random variables is replaced by branching mechanism for point processes. 
We expect that similarly, this family shall describe attractors for natural dynamics involving branching mechanisms.

To present more precisely our purpose and results, it is convenient to set up right now some notation that we will use throughout this work. 
We write $\calP$ for the space of non-decreasing sequence $\x = (x_k)_{k \geq 1}$ in $(-\infty,\infty]$ with $\lim_{k \to \infty} x_k = \infty$, and endow $\calP$ with the topology of pointwise convergence. This space is canonically identified with the space of locally finite point measures giving finite mass to $(-\infty,0]$ by
\[
  \x \simeq \sum_{k \geq 1} \delta_{x_k},
\]
with the convention $\delta_\infty = 0$. In other words, the $k$-th element $x_k$ of $\x$ is viewed as the location of 
the $k$-th left-most atom of a locally finite point measure, with the convention that $x_k=\infty$ if that point measure has total mass less than $k$. We shall view thus $\x \in \calP$ both as a non-decreasing sequence and as a point measure, often going from one interpretation to the other without explicit mention of it.

The following two basic operators on $\calP$ will play an important role in the sequel. First, the translation operator with displacement $y\in\R$ simply transforms $\x=(x_k)_{ k\geq 1}$ into $y+\x\coloneqq (y+x_k)_{k\geq 1}$. Second, the dilation operator with a factor $c>0$, transforms $\x$ into $c\x\coloneqq(cx_k)_{ k\geq 1}$.

We also equip $\calP$ with an internal composition law $\sqcup$ corresponding to superposition; specifically $\x \sqcup \y $ denotes the sequence in $\cal P$ obtained by arranging the elements of $\x$ and $\y$ (elements are repeated according to their multiplicities) in the non-decreasing order. When we think of $\x$ and $\y$ as locally finite point measures, $\x \sqcup \y $ then simply corresponds to the sum\footnote{However, in order to avoid a possible confusion with the translation operator, we prefer to use the notation $\x \sqcup \y $ rather than $\x+\y$. Beware also that $c\x$, viewed as a measure, is different from assigning mass $c$ to each atom of $\x$ and rather corresponds to dilating the location of each atom by a factor $c$.} of the two. 

Next, we denote the integral of a function $g:\R\to \R_+$ with respect to $\x$ by
$$\crochet {\x, g}\coloneqq \int_{(-\infty, \infty)} g(x) \x(\dd x)=\sum_{k\geq 1} g(x_k).$$
In the case when $g = \mathbf{1}_{A}$ is the indicator function of some subset $A\subseteq \R$, we simply write $\x(A)\coloneqq \crochet {\x, \mathbf{1}_{A}}$ for the number of atoms of $\x$ in $A$. 

In this framework, a \emph{branching random walk} $\X := (\X_n)_{n \geq 0}$ is a $\calP$-valued random process in discrete time, started from a single atom at $0$ (i.e. we have $\X_0= (0,\infty,\infty,\ldots)$), which fulfills the following branching property. There exists a sequence $(\X^{(k)})_{k\in\N}$ of i.i.d. copies of $\X$, which are further independent of $\X_1$, such that for all $n\geq 0$, we have
$$\X_{n+1}=\bigsqcup_{k\geq 1} (x_k+\X^{(k)}_n),\qquad \text{with }(x_k)_{k\geq 1}=\X_1.$$
We call the law of $\X_1$ (and sometime, by a slight abuse, $\X_1$ itself) the \emph{reproduction law} of the branching random walk $\X$. The lecture notes \cite{Shi15} provide a self-contained introduction to this topic.

We are interested here in the sub-family of reproduction laws satisfying an additional self-similarity property, which is analogous to scaling for random variables (and therefore rather use for it the notation $\S$).

\begin{definition}\label{def_branching_stable} Let $\S_1$ be a random point measure in $\calP$
such that 
\begin{equation}\label{e:nondegLap}
\E
\left(
\crochet{\S_1, \e_{-\vartheta}}
\right)<\infty \quad \text{for some $\vartheta\geq 0$,}
\end{equation}
where $ \e_{-\vartheta}(x)=\e^{-\vartheta x}$. We call $\S_1$ 
 \emph{branching-stable}
if it fulfills the following property: there exists a sequence of positive real numbers $(a(n))_{n\geq 1}$ such that for each $n\geq1$, there is the identity in distribution
\begin{equation}
  \label{eqn:defStability}
 \S_n \egaldistr  a(n) \S_1, 
\end{equation}
where $(\S_n)_{n\geq 0}$ denotes a branching random walk with reproduction law $\S_1$.
\end{definition}

\begin{remark}
This definition is different from that in \cite{ZaZ15}, where a random point measure $\Y$ is called branching-stable with exponent $\alpha >0$ if, in the present notation, for all $t\in(0,1)$, there is the identity in distribution
\[
 \left( t^{1/\alpha} \bullet \Y_1  \right) \sqcup \left( (1-t)^{1/\alpha} \bullet \Y_2  \right) \egaldistr \Y ,
\]
where $\Y_1$ and $\Y_2$ are independent copies of $\Y$, and for $s\in(0,1)$, $s \bullet \Y$ represents a point measure such that every atom in $\Y$ is replaced by the value at time $-\log s$ of an i.i.d. copy of a continuous-time Markov branching process. 
\end{remark}

Condition~\eqref{e:nondegLap} in Definition~\ref{def_branching_stable} is a non-degeneracy requirement for the Laplace transform of the intensity measure which is fairly standard for branching processes and lies at the heart of many useful tools in this area. In particular, it is essential in the article \cite{BeMa}, which shall also play an important role here.

Two degenerate examples of random point measures fulfilling \eqref{e:nondegLap} and \eqref{eqn:defStability} are $\S_1=(\infty, \ldots)=0$ a.s. and $\S_1 =(0,\infty,\infty,\ldots) = \delta_0$ a.s. (then the normalizing sequence $(a(n))_{ n \geq 1}$ can be chosen arbitrarily). We henceforth only consider non-degenerate
cases. There are further obvious examples with a trivial branching mechanism, that are associated to stable L\'evy processes. Specifically, consider for some $\alpha\in(0,2)$ a random walk $(S_n)_{n \geq 0}$ with step distribution given by an $\alpha$-stable random variable, and set $\S_n = \delta_{S_n}$. Then $(\S_n)_{n \geq 0}$ can be viewed as a branching random walk with a unique offspring for each atom at each generation. Obviously \eqref{e:nondegLap}  holds with $\vartheta=0$ and also
\eqref{eqn:defStability} with $a(n) = n^{1/\alpha}$. 

The starting point of our work is the observation that the family of bran-ching-stable point measures is actually quite large. To start with, we observe in Lemma~\ref{l:BM1} that the normalizing sequence in \eqref{eqn:defStability} has necessarily the form $a(n)=n^{\varepsilon/\alpha}$ with $\varepsilon = \pm1$ and $\alpha >0$. We call $\varepsilon \alpha$ the \emph{scaling exponent}. We shall prove in Proposition~\ref{prop:oneParticle} that the case of a positive scaling exponent is always trivial, in the sense that a branching-stable point measure for which $\varepsilon =+1$ is always given by a single atom whose location has an $\alpha$-stable distribution. In particular, the only non-degenerate cases with a positive scaling exponent occur for $0<\alpha\leq 2$.
The family of branching-stable point measures with negative scaling exponent turns out to be much more interesting, as we shall now explain. 

Recall first from \cite{BeMa}\footnote{Beware however that in \cite{BeMa}, we considered point measures whose atoms can be ranked in {non-increasing} order rather than in non-decreasing order as here, so the results taken from there are rephrased after an obvious reflexion $\x\mapsto -\x$.} that a random point measure $\X_1$ satisfying \eqref{e:nondegLap} is called \emph{infinitely ramified} if for every $n\geq 1$, it has the same distribution as the $n$-th generation of some branching random walk. In other words, for each $n \geq 1$, there exists a branching random walk $\X^{(n)}=(\X^{(n)}_k)_{k\geq 0}$ such that $\X_1\egaldistr \X^{(n)}_n$. Infinite ramification can thus be seen as a branching version of the notion of infinitely divisibility. In this direction recall {\it e.g.} from \cite{Sato}, that every infinitely divisible variable $\xi_1$ arises as the value at time $t=1$ of a Lévy process $(\xi_t)_{t\geq 0}$. A similar connection was established in \cite{BeMa} between infinitely ramified point measures and so-called \emph{branching Lévy processes}. Roughly speaking, a branching Lévy process is a càdlàg process in continuous time with values in $\calP$, whose atoms move (independently of one-another) according to a certain Lévy process, and give birth to children in a Poissonian fashion; we stress that the total branching rate may be infinite. We refer to Section~1 in \cite{BeMa} for a precise definition and a characterization of such processes, and to Section~5 of the same paper for a rigorous construction. It is proven in that article, that every infinitely ramified point measure $\X_1$ can be obtained as the value at time $t=1$ of a branching Lévy process $(\X_t)_{ t \geq 0}$.

Any branching-stable point measure is plainly infinitely ramified, and hence we essentially have to determine the class of branching L\'evy processes which fulfill the scaling property.
In Theorem~\ref{T1}, we prove that branching-stable point measures with negative scaling exponent are associated to branching L\'evy processes without spatial displacement (i.e. individuals are static and eternal) and for which the so-called L\'evy measure (which characterizes the statistics of reproduction events) is self-similar. More explicitly, the latter can be viewed as a special class of Crump-Mode-Jagers branching processes (see Crump and Mode \cite{CrumpMode}, Jagers \cite{Jagers}; we shall write simpy CMJ branching processes in the sequel) in which individuals beget progenies according to a time-homogeneous Poisson point process with a scale-invariant intensity measure, and such that all children are located at the right of their parents. 
We also provide an equivalent description in terms of a branching random walk $(\Z_n)_{n\geq 0}$ in $\R_+^2$,
where atoms record the birth times and locations of individuals in the CMJ process.

In the second part of the article, we present some properties of branching-stable point measures with negative scaling exponents. Relying on classical results due to Biggins \cite{Biggins77,Biggins92}, we investigate the asymptotic behavior of the cumulative distribution function $\S_1((-\infty,x])$ as $x \to \infty$. We establish in Theorem~\ref{T2} the convergence in distribution of this normalized quantity
\begin{equation}
 \label{eqn:t2forDummies}
 \lim_{x \to \infty} \frac{\S_1((-\infty,x])}{\E\left( \S_1((-\infty,x])\right)} = \bar {W} \quad \text{ in law,}
\end{equation}
where $\bar{W}$ is the limit of the so-called Biggins' additive martingale. It turns out that the mean $\E\left( \S_1((-\infty,x])\right)$ can be computed explicitly in terms of the so-called Wright generalized Bessel function, and in particular 
its asymptotic expansion as $x\to \infty$ is already known in the literature. 

Finally, we study the genealogy of the atoms of $\S_1$, from the viewpoint of CMJ processes. We shall describe the ancestral lineage of typical atoms, as well as the asymptotic position of the left-most atom at the $n$-th generation as $n\to \infty$. In this direction, information due to Biggins \cite{Big78} about the shape of the associated $2$-dimensional branching random walk $(\Z_n)_{n\geq 0}$ plays a crucial role.

\paragraph*{Organization of the paper.} The rest of the article is organized as follows. In Section~\ref{sec:structure}, we unveil the fine structure of branching-stable point processes, and in particular prove the main result of the article, namely, Theorem~\ref{T1}, which characterizes the law of branching-stable point measure with negative scaling exponent in terms of CMJ branching processes. In Section~\ref{sec:martingale}, we study in more details the cumulative distribution function of branching-stable point measures, obtaining the convergence in law for this process. Finally, in Section~\ref{sec:generation}, we turn our attention to the genealogy of atoms.

\section{The structure of branching-stable point measures}
\label{sec:structure}

The purpose of this section is to describe precisely the structure of branching-stable point measures. We implicitly exclude the degenerate case\footnote{In short, if either $\S_1$ is empty a.s., or reduced to a single atom at $0$ a.s., then \eqref{eqn:defStability} and \eqref{e:nondegLap} hold plainly. On the other hand, if all the atoms of $\S_1$ are located at $0$ a.s., but the number of atoms is neither equal to $0$ a.s. or $1$ a.s., then \eqref{eqn:defStability} must fail.} when all the atoms are located at $0$ a.s. This enables us to assert that the real number $a(n)>0$ in \eqref{eqn:defStability} is unique (for instance, by considering the distribution of the smallest non-zero atom). As a consequence, there is the identity $a(nk)=a(n)a(k)$ for all integers $n,k\geq 1$.

Our analysis crucially relies on \cite{BeMa}, in which it is shown that every infinitely ramified point measure can be viewed as  some branching Lévy processes evaluated at time $t=1$. A branching-stable point measure is \emph{a fortiori} infinitely ramified, and we specialize some results of \cite{BeMa} in the present setting.

\begin{lemma} \label{l:BM1} Let $\S_1$ be a random point measure that fulfills \eqref{e:nondegLap} and \eqref{eqn:defStability}. Then there exist: 

\noindent $\bullet$ a c\`adl\`ag process in continuous time $(\S_t)_{t\geq 0}$ with values in $\calP$, started from $\S_0=\delta_0$ and taking the value $\S_1$ at time $t=1$; 

\noindent $\bullet$ a real number $\varepsilon \alpha$ with $\alpha >0$ and $\varepsilon=\pm1$; 

\noindent such that the following assertions hold.
\begin{enumerate}
 \item [(i)] The process $\S=(\S_t)_{t\geq 0}$ fulfills the branching property: for every $r>0$, there exists a sequence $\S^{(k)}$ of i.i.d. copies of $\S$ such that 
 $$\S_{t+r}=\bigsqcup_{k\geq 1} (x_k+\S^{(k)}_t),\qquad \text{with }(x_k)_{k\geq 1}=\S_r,$$
 for every $t\geq 0$.
  \item [(ii)] The process $(\S_t)_{t\geq 0}$ is self-similar with scaling exponent $\varepsilon \alpha$, in the sense that for every $c >0$, there is the identity in distribution
$$\left( \S_{c^{\varepsilon \alpha}t} \right)_{t\geq 0} \egaldistr \left( c\S_t \right)_{t\geq 0}.$$ 
\end{enumerate}
\end{lemma}

\begin{remark}
We stress that the law of the process $(\S_t, t \geq 0)$ in Lemma~\ref{l:BM1} is uniquely determined by the law of $\S_1$. Indeed, its one-dimensional marginals are characterized by (ii), and therefore, by the branching property (i), the finite dimensional marginals are uniquely determined as well. The fact that the process is c\`adl\`ag enables us to conclude. This contrasts with \cite{BeMa}, in which the question of uniqueness of the branching Lévy process associated to an infinitely ramified point measure is still pending.
\end{remark}

\begin{proof}
Theorem 1.1 in \cite{BeMa} ensures the existence of a c\`adl\`ag process $(\S_t)_{t\geq 0}$ that satisfies (i). More precisely, the argument relying on Kolmogorov's extension theorem in the proof of Proposition 2.1 of \cite{BeMa} combined with \eqref{eqn:defStability} shows actually that we can construct $(\S_t)_{t\geq 0}$ such that for every integer $n\geq 0$
\[a(2^n) \S_{2^{-n}} \egaldistr \S_1.\]
Then \eqref{eqn:defStability} yields for every integer $k\geq 1$
$$ \S_{k2^{-n}} \egaldistr \frac{a(k)}{a(2^n)}\S_1.$$
As a consequence we can define unambiguously $a(k2^{-n})=a(k)/a(2^{n})$ for dyadic rational numbers, and from the right-continuity of the process $(\S_t)_{t\geq 0}$, one readily deduces that $t\mapsto a(t)$ has a right-continuous extension to all real numbers $t>0$. Plainly, the function $a$ is multiplicative, thus there exists some exponent $\gamma\in\R$ such that $a(t) = t^{\gamma}$. 

We then note that the case $\gamma=0$ would yield a degenerate branching-stable point measure. Indeed, this is seen from letting $n\to \infty$ in the identity in distribution $\S_{2^{-n}}\egaldistr \S_1$ and using the facts that $(\S_t)_{t\geq 0}$ has c\`adl\`ag paths and that $\S_0$ reduces to a single atom at $0$ a.s.

Therefore $\gamma\neq 0$ and we conclude that (iii) holds with $1/\gamma$ replaced by $\varepsilon \alpha$, first in the sense of finite dimensional distributions, and then in the sense of processes since the sample paths are c\`adl\`ag a.s. The proof is complete.
\end{proof}

We call any process $(\X_t)_{t\geq 0}$ with c\`adl\`ag paths in $\calP$ such that the non-degeneracy condition for the Laplace transform \eqref{e:nondegLap} holds for $\X_1$, and which fulfills the branching property (i) and the scaling property (ii) of Lemma~\ref{l:BM1}, a \emph{branching-stable process} with scaling exponent $\varepsilon \alpha$. Then plainly $\X_1$ is a branching-stable point measure, and there is thus a one-to-one and onto correspondence between distributions of branching-stable point measures and distributions of branching-stable processes. 

Upon the simple reflexion $\S_t\mapsto -\S_t$, branching-stable processes can be viewed as a special class of \emph{branching L\'evy processes} as defined in \cite{BeMa}, and it has been shown there that the distribution of the latter is determined by a triplet, analogous to the characteristic triplet for L\'evy processes. It would thus be natural now to identify the subclass of characteristic triplets of branching L\'evy processes that correspond to branching-stable processes. However we shall rather follow a slightly different road for which the scaling property is easier to exploit, and which yields a simpler representation of branching-stable processes than that the general construction in \cite{BeMa}. 

The structure of branching-stable processes depends crucially on the sign $\varepsilon$ of the scaling exponent. More precisely, we show that if $\varepsilon = +1$, then the branching-stable process has a trivial branching mechanism (no death or reproduction event occurs). If $\varepsilon=-1$, then reproduction events occur, but the particles do not move. We first treat the case when $\varepsilon=+1$.

\begin{proposition}
\label{prop:oneParticle}
Assume that $\varepsilon=+1$.
\begin{enumerate}
\item[(i)]
If $\alpha >2$, then $\S_1$ is degenerate. 
\item[(ii)]
If $0<\alpha\leq 2$, then there exists an $\alpha$-stable L\'evy process $(\xi_t)_{t\geq 0}$
such that $\S_t$ is reduced to a single atom at $\xi_t$, viz. $\S_t=\delta_{\xi_t}$ for all $t\geq 0$, a.s.
\end{enumerate}
\end{proposition}

\begin{proof}
By the self-similarity property of Lemma~\ref{l:BM1}(ii), we have for all $c>0$
$$\S_{c^{\alpha}}([-1,1]) \egaldistr (c\S_1)([-1,1]) = \S_1([-1/c,1/c]).$$
Letting $c\to 0+$ and recalling that $\S_0=\delta_0$, we see that $\S_1(\R)=1$ a.s., meaning that $\S_1$ is reduced to a single atom. Plainly, the same holds more generally for $\S_t$ for all $t\geq0$. If we write $\xi_t$ for the location of the unique atom of $\S_t$, then the branching property of the process $(\S_t)_{t\geq 0}$ translates into independence and stationarity of the increments of $(\xi_t)_{t\geq 0}$. The latter further inherits c\`adl\`ag paths and self-similarity from $(\S_t)_{t\geq 0}$, so $(\xi_t)_{t\geq 0}$ must be an $\alpha$-stable L\'evy process. Since the latter are degenerate when $\alpha>2$, this completes the proof. \end{proof} 
 
Our purposes in the rest of this section are, first to provide a construction of fairly natural examples of branching-stable processes with a negative scaling exponent $-\alpha$, and second, to show that actually any branching-stable process with scaling exponent $-\alpha$ can be obtained by this construction. In this direction, we start by introducing some notation.
 
We write $\calP^*_+$ for the space formed by the sequences $\x=(x_k)_{ k\geq 1}\in \calP$ with $x_k\in(0,\infty]$ for all $k\geq 1$ and $x_1<\infty$. The building block of our construction consists of a finite measure $\lambda\neq 0$ on $\calP^*_+$ such that
\begin{equation}
 \label{condlambda}
 c(\lambda)\coloneqq \int_{\calP^*_+} \sum_{k\geq 1} x_k^{-\alpha} \lambda(\dd \x)<\infty .
\end{equation}
We first define a sigma-finite measure $\Lambda^*$ on $\calP^*_+$, such that for every measurable functional $F: \calP^*_+\to \R_+$,
\begin{equation}
 \label{defLambda}
 \int_{\calP^*_+} F(\x) \Lambda^*(\dd \x)\coloneqq \int_0^{\infty} y^{\alpha-1} \int_{\calP^*_+} F(y\x) \lambda(\dd \x) \dd y .
\end{equation}

Next consider a Poisson point process $\NN$ on $(0,\infty)\times \calP^*_+$ with intensity $\dd t\times \Lambda^*(\dd \x)$. Viewing each atom $(t,\x)$ of $\NN$ as a sequence of atoms $((t,x_k))_{k\geq 1}$ on the fiber $\{t\}\times (0,\infty)$ (and ignoring atoms $(t,\infty)$ if any), $\NN$ induces a point process $\Z_1$ on $(0,\infty)^2$ defined by
$$\crochet{\Z_1,f}\coloneqq \int_{(0,\infty)\times \calP^*_+} \crochet{\x,f(t,\cdot)} \NN(\dd t, \dd \x),$$
where $f$ is a generic measurable nonnegative function on $(0,\infty)^2$ and 
$$\crochet{\x,f(t,\cdot)}= \sum_{k\geq 1} f(t,x_k).$$

Note from elementary Poissonian calculus that if we write $\mu_1$ for the intensity measure of $\Z_1$, then for every $t,a>0$, we have
\begin{eqnarray*} \mu_1((0,t]\times (0,a])&=&
\E(\Z_1((0,t]\times (0,a]))\\
&=& t \int_{\calP^*_+}\sum_{k\geq 1} \ind{x_k\leq a} \Lambda^*(\dd \x)\\
&=&t \int_{\calP^*_+}\sum_{k\geq 1} \int_0^{\infty} \ind{yx_k\leq a} y^{\alpha-1} \dd y \lambda(\dd \x)\\
&=&\alpha^{-1}t a^{\alpha}\int_{\calP^*_+}\sum_{k\geq 1} x_k^{-\alpha} \lambda(\dd \x)\\
&=&\alpha^{-1} c(\lambda) t a^{\alpha}.
\end{eqnarray*}
We thus get the identity
\begin{equation}\label{e:intens1}
\mu_1(\dd t, \dd x) = c(\lambda) x^{\alpha-1} \dd t \dd x, \qquad (t,x)\in(0,\infty)^2.
\end{equation}

We now view $\Z_1$ as the point process describing the reproduction of an individual in a 
CMJ branching process in $\R_+$. That is, an atom of $\Z_1$ at $(t,x)\in\R_+^2$ is interpreted as a birth event occurring when the age of the parent is $t$ and where the child is located at distance $x$ at the right of its parent. Note that the parent may beget several children at the same age, this corresponds to the situation where $(t,\x)$ is an atom of $\NN$ with $\x\in\calP_+^*$ having more than one atom.

The construction of this CMJ branching process can be realized as follows using a branching random walk $(\Z_n)_{n\geq 0}$ on $\R_+^2$ with reproduction law $\Z_1$. We consider the point measure $\bigsqcup_{n\geq 0} \Z_n$ that consists of all the atoms $(t,x)$ appearing in the branching random walk $(\Z_n)_{n\geq 0}$, possibly repeated according to their multiplicities; see Figure~\ref{fig:simul} for a simulation. We interpret this point measure on $\R_+^2$ as a population of static individuals living in $\R_+$, which grows as time passes. An atom $(t,x)$ of $\bigsqcup_{n\geq 0} \Z_n$ marks the birth of a new individual at time $t$ and position $x$. Once an individual is born, it remains at its birth position forever. For every $t\geq 0$, we write $\X_t$ for the point process on $\R_+$ that represents the locations of particles alive at time $t$, viz.
\begin{equation} \label{def_X}
\crochet{\X_t, g}= \sum_{n=0}^{\infty}\int_{[0,t]\times \R_+}
g(x)\Z_n(\dd s, \dd x),
\end{equation}
where $g:\R_+\to \R_+$ stands for a generic measurable function, and of course $\Z_0=\delta_{(0,0)}$ as usual.
\begin{figure}
\begin{center}
\includegraphics[scale=0.5, width=6cm]{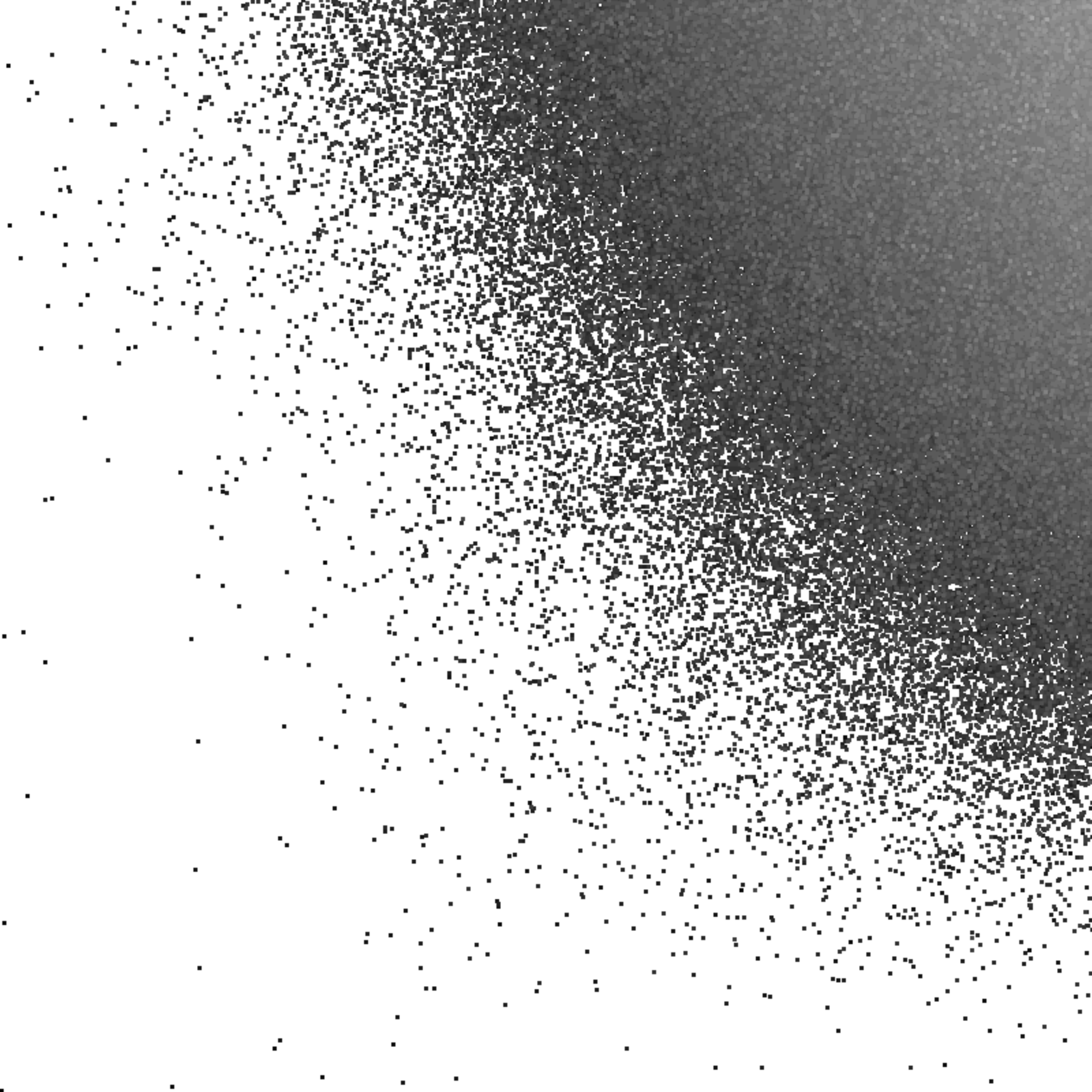}
\end{center}
\caption{A simulation of a random point measure $\bigsqcup_{n\geq 0} \Z_n $ on $\R_+\times \R_+$ with self-similarity exponent $-1$.}\label{fig:simul}
\end{figure}

\begin{theorem}\label{T1} \begin{enumerate}
\item[(i)] The process $(\X_t)_{t\geq 0}$ constructed above is branching-stable with scaling exponent $-\alpha$.
\item[(ii)] Conversely, for every branching-stable process with scaling exponent $-\alpha$, say $(\S_t)_{t\geq 0}$,
there exists a finite measure $\lambda$ on $\calP_+^*$ satisfying \eqref{condlambda}, such that
$(\X_t)_{t\geq 0}$ and $(\S_t)_{t\geq 0}$ have the same law.
\end{enumerate}
\end{theorem}

We stress that different finite measures $\lambda$ on $\calP^*_+$ may yield via \eqref{defLambda} the same measure $\Lambda^*$ and hence the same branching-stable process $(\X_t)_{t\geq 0}$ (in law); however, if we further impose that $\lambda(\{\x\in \calP_+^*: x_1\neq 1\})=0$, then 
uniqueness of $\lambda$ follows; see the forthcoming Lemma~\ref{lem_triplet}.

The proof of the first part of Theorem~\ref{T1} is rather easy:
\begin{proof}[Proof of Theorem~\ref{T1}(i)] The key point in the definition \eqref{defLambda} of $\Lambda^*$ is that for every $c>0$ and measurable functional $F: \calP^*_+\to \R_+$, one has
\begin{eqnarray*}
\int_{\calP^*_+} F(c\x) \Lambda^*(\dd \x)&=& \int_0^{\infty} y^{\alpha-1} \int_{\calP^*_+} F(cy\x) \lambda(\dd \x) \dd y \\
&=& c^{-\alpha} \int_0^{\infty} r^{\alpha-1} \int_{\calP^*_+} F(r\x) \lambda(\dd \x) \dd r \\
&=& c^{-\alpha} \int_{\calP^*_+} F(\x) \Lambda^*(\dd \x).
\end{eqnarray*}
In words, the image of $\Lambda^*$ by the operator which dilates locations of atoms by a factor $c$ is $c^{-\alpha}\Lambda^*$. 
By the mapping theorem for Poisson point processes, the image of $\NN$ on $\R_+\times \calP^*_+$ by the map $(t,x)\mapsto (c^{-\alpha}t, c\x)$ is also Poisson with intensity $\dd t \times \Lambda^*(\dd x)$, and thus has the same distribution as $\NN$. 

As a consequence, the distribution of the point process $\Z_1$ on $\R_+\times \R_+$ is invariant for the map $(t,x)\mapsto (c^{-\alpha} t, cx)$, and then the same also holds for the branching random walk $(\Z_n)_{n\geq 0}$. A fortiori, the law of the point process $\sum_{n\geq 0} \Z_n$ is invariant for the map $(t,x)\mapsto (c^{-\alpha} t, cx)$, and this entails the self-similarity of the process $(\X_t)_{t\geq 0}$. 

The branching property of $(\X_t)_{t\geq 0}$ is readily seen from the interpretation of $(\X_t)_{t\geq 0}$ as a CMJ branching process. Recall that individuals are static and beget children around their position according to a Poisson point process with intensity $\dd t \times \Lambda^*(\dd \x)$. In this setting, the translation invariance of the Lebesgue measure $\dd t$ entails that the reproduction process of individuals is also invariant by time shift (meaning that the reproduction is age-independent), and the branching property of $(\X_t)_{t\geq 0}$ follows from the restriction theorem of Poisson point measures. 
 
It remains to verify that \eqref{e:nondegLap} holds for $\S_1=\X_1$. In this direction, note from \eqref{e:intens1} that for every $r>0$, there is the identity
$$\E\left(\int_{(0,\infty)^2} \e^{-rs- x} \Z_1(\dd s, \dd x) \right) = r^{-1}c(\lambda) \Gamma(\alpha).$$
We choose $r>c(\lambda) \Gamma(\alpha)$, so that the geometric series $\sum_{n\geq 0}( r^{-1}c(\lambda) \Gamma(\alpha))^n$ converges. This entails that 
$$\int_{(0,\infty)} \e^{-rt} \E\left(\crochet{\X_t,\e_{-1}} \right) \dd t
= r \sum_{n\geq 0} \E\left(\int_{(0,\infty)^2} \e^{-rs- x} \Z_1(\dd s, \dd x) \right) <\infty,$$
and thus \eqref{e:nondegLap} holds for $\vartheta =1$ (and then more generally for any $\vartheta >0$, by scaling). 
\end{proof}

The reflected branching-stable process $(-\X_t)_{t\geq 0}$ is a branching L\'evy process in the terminology of \cite{BeMa} (recall that reflexion is needed to fit the framework there). More precisely, comparing the construction here and that in \cite{BeMa} actually identifies its characteristics $(\sigma^2,a,\Lambda)$: the Gaussian coefficient is $\sigma^2=0$, the drift coefficient is $a=0$, and the so-called L\'evy measure $\Lambda$ is simply the image of $\Lambda^*$ by the map $\x\mapsto (0, -x_1, -x_2, \ldots)$ (recall that parents survive at every birth event and remain at the same location, which explains that the first coordinate is always $0$).

In particular, for every $\theta>0$, we have
\begin{eqnarray*}
\int_{\calP^*_+} \sum_{i=2}^{\infty} \e^{\theta y_i} \Lambda(\dd \y)
&=& \int_{\calP^*_+} \sum_{k=1}^{\infty} \e^{-\theta x_k} \Lambda^*(\dd \x)\\
&=& \int_0^{\infty} y^{\alpha-1} \int_{\calP^*_+} \sum_{k=1}^{\infty} \e^{-y\theta x_k} \lambda(\dd \x)\\
&=&\theta^{-\alpha}\Gamma(\alpha) \int_{\calP^*_+} \sum_{k=1}^{\infty} x_k^{-\alpha} \lambda(\dd \x)\\
&=& c(\lambda) \Gamma(\alpha) \theta^{-\alpha}.
\end{eqnarray*}
As a consequence, we see that the requirement (1.4) as well as the more demanding one of Corollary 6.8 in \cite{BeMa} are satisfied. 

One could establish Theorem~\ref{T1}(ii) by showing that conversely, the characteristics $(\sigma^2,a,\Lambda)$ of any branching L\'evy process that is further self-similar with a negative scaling exponent are necessarily given by $\sigma^2=a=0$ and a L\'evy measure $\Lambda$ as above. This would require working with censoring techniques developed in \cite{BeMa}, and would make the interpretation of branching-stable processes as CMJ branching processes less clear, so we shall rather follow a slightly different approach.

The first step of the proof of Theorem~\ref{T1}(ii) consists in the following observation. 
\begin{lemma}
\label{lem:lazyParent}
Let $(\S_t)_{t\geq 0}$ be a branching-stable process with negative scaling exponent $-\alpha$. For all $t > 0$, we have
\[
 \S_t(\{0\}) = 1 \quad, \quad \S_t((-\infty,0)) = 0 \quad \text{and} \quad \S_t((0,\infty)) = \infty\quad \text{a.s.} 
\]
\end{lemma}

\begin{proof} Just as in the proof of Proposition~\ref{prop:oneParticle}, we have for all $a>0$
$$\S_{a^{-\alpha}}([-1,1]) \egaldistr \S_1([-1/a,1/a]).$$
Letting now $a\to \infty$, we see that $\S_1(\{0\})=1$ a.s., meaning that $\S_1$ has always a unique atom at $0$, and by self-similarity, this extends to $\S_t$ for all $t\geq 0$. 

The second assertion of the statement is proved by contradiction. Suppose that $\P(\S_t((-\infty,0))\geq 1)=p>0$ for some $t>0$. Then by self-similarity, the same holds for all $t>0$, and thus $\P(\S_t((-\infty,0]))\geq 2)=p>0$. We take $t=2^{-n}$ and deduce from Lemma~\ref{l:BM1}(ii) that $\S_1((-\infty,0])$ is stochastically bounded from below by the size of the $2^n$-th generation of a Galton-Watson process with reproduction law $\nu$ with $\nu(0)=0$ and $\nu(1)\leq 1-p$. It follows that $\S_1((-\infty,0])=\infty$ a.s., which contradicts the fact that $\S_1\in\calP$ a.s.

Recall that the degenerate case $\S_1=\delta_0$ a.s. has been excluded, and therefore 
$\P(\S_1((0,\infty))>0)>0$. Then essentially the same argument as above shows that we must have $\S_t((0,\infty))=\infty$ a.s. for all $t>0$.
\end{proof}

Lemma~\ref {lem:lazyParent} shows that $\S_t$ always consists of a single atom at $0$ and a non-decreasing sequence of atoms in $(0,\infty)$. It incites us to view that atom at $0$ as the ancestor of a population and those in $(0,\infty)$ as an element of $\calP_+^*$ describing the descent of that ancestor at time $t$. That is, using the branching property, we view $(\S_t)_{t\geq 0}$ as a population model in $[0,\infty)$, such that individuals are eternal and static, and beget progeny located at their right as time passes. In this direction, we recall the natural genealogical structure of the branching-stable process as described in Section 6.1 in \cite{BeMa}. Roughly speaking, for every $t\geq 0$ and every $j\geq 1$, one can define a c\`adl\`ag path $x_{j,t}: [0,t] \to [0,\infty)$, where $x_{j,t}(s)$ shall be viewed as the ancestor at time $s$ of the $j$-th atom of $\S_t$ (in particular the terminal value $x_{j,t}(t)$ is the $j$-th atom of $\S_t$). In other words, $x_{j,t}$ describes the ancestral lineage of the $j$-th atom of $\S_t$.

\begin{lemma}
\label{l:BM2}
For every $t\geq 0$ and every $j\geq 1$, the ancestral lineage $x_{j,t}$ is a non-decreasing step function a.s. 
\end{lemma}

\begin{proof}
We recall the pathwise many-to-one formula (\cite[Lemma 6.4]{BeMa}): For every non-negative measurable functional $f$ on the space of c\`adl\`ag paths on $[0,t]$, we have 
\[
 \E\left( \sum_{j \geq 1} f(x_{j,t}(s): 0\leq s \leq t) \right) = \E\left( \e^{-\theta \xi_t + t \kappa(\theta)} f(\xi_s: 0\leq s \leq t) \right),
\]
where $(\xi_s)_{s\geq 0}$ is a certain L\'evy process and 
$\kappa(\theta) = \log \E(\crochet{\S_1, \e_{- \theta}})$ (recall the non-degeneracy condition \eqref{e:nondegLap}). 

In this setting, Lemma~\ref{lem:lazyParent} implies that $\xi_t\geq 0$ a.s. and $\P(\xi_t=0)>0$, so $(\xi_s)_{s\geq 0}$ must be a compound Poisson process in $[0,\infty)$. In particular, its paths are non-decreasing step functions a.s. We complete the proof using the pathwise many-to-one formula with $f(x) = 1-\ind{x \text{ is a non-decreasing step function}}$ and the union bound.
\end{proof}

Lemma~\ref{l:BM2} enables us to define a birth-time and a generation for each atom of $\S_t$. Specifically,
the birth-time of the $j$-th atom of $\S_t$ is the instant $b_{j,t}\leq t$ at which its ancestral lineage 
makes its last jump,
$$b_{j,t}\coloneqq \inf\{s\leq t: x_{j,t}(s)=x_{j,t}(t)\},$$
and its generation $g_{j,t}$ is the number of jumps of its ancestral lineage,
$$g_{j,t}\coloneqq \#\{0<s\leq t: x_{j,t}(s-)< x_{j,t}(s)\}.$$
Plainly, if $t'>t$, the $j$-th atom of $\S_t$ is also the $j'$-th atom of $\S_{t'}$ for some $j'\geq j$, and then
the ancestral lineage $x_{j',t'}$ remains constant on the time interval $[t,t']$, so the birth-time and generation remain of course unchanged. 

We can now define for every $n\geq 0$ the point process $\Y_n$ in $\R_+^2$ whose atoms are given by the birth-times and locations of the atoms with generation $n$ that appear in the branching-stable process.
Specifically, for every $t\geq 0$ and $B\in {\mathcal B}(\R_+)$, we set
$$\Y_n([0,t]\times B)\coloneqq \#\{j\geq 1: g_{j,t}=n\ \& \ x_{j,t}(t)\in B\}.$$
Observe that for every $t\geq 0$, 
one recovers $\S_t$ as
$$\S_t(B)=\sum_{n\geq 0} \Y_n([0,t]\times B), \qquad B\in {\mathcal B}(\R_+),$$
so the proof of Theorem~\ref{T1}(ii) is now reduced to establishing that the process $(\Y_n)_{n\geq 0}$
is actually a branching random walk, and that its reproduction law $\Y_1$ is of the same type as $\Z_1$
for an appropriate choice of $\lambda$.

\begin{lemma}\label{lem:brwplan} The process $(\Y_n)_{n\geq 0}$ is a branching random walk in $\R_+^2$.
\end{lemma}

\begin{proof} We shall establish the first claim by approximation based on discrete time skeletons.
For every fixed $\ell\geq 0$, we 
consider the restriction of the branching-stable process to times in $D_{\ell}\coloneqq \{k2^{-\ell}: k\geq 0\}$. Clearly $(\S_t)_{t\in D_{\ell}}$ is a branching random walk in discrete time, whose ancestral lineages are simply obtained from those of $(\S_t)_{t\geq 0}$ by restriction to $D_{\ell}$.

For $t\in D_{\ell}$, we write $b^{(\ell)}_{j,t}$ for the birth-time of the $j$-th atom of $\S_t$ and $g^{(\ell)}_{j,t}$ for its generation 
as viewed from the skeleton $(\S_t)_{t\in D_{\ell}}$, that is
$$b^{(\ell)}_{j,t}\coloneqq \inf\{s\leq t, s\in D_{\ell} : x_{j,t}(s)=x_{j,t}(t)\} = 2^{-\ell}\lceil 2^{\ell} b_{j,t}\rceil,$$
and 
$$ g^{(\ell)}_{j,t}\coloneqq \#\{s< t, s\in D_{\ell} : x_{j,t}(s) < x_{j,t}(s+2^{-\ell})\}.$$
We also write $\Y_n^{(\ell)}$ for the random point measure such that for every $t\in D_{\ell}$ 
and $B\in {\mathcal B}(\R_+)$:
$$\Y^{(\ell)}_n([0,t]\times B)= \#\{j\geq 1: g^{(\ell)}_{j,t}=n\ \& \ x_{j,t}(t)\in B\}.$$

In particular, $\Y^{(\ell)}_0=\delta_{(0,0)}$ and $\Y^{(\ell)}_1$ is the point process induced by the birth-times and locations of individuals of the first generation. These form an optional line in the terminology of Jagers \cite{Jagers} at which the strong branching property holds (Theorem 4.14 in \cite{Jagers}). By iteration, we get that $(\Y^{(\ell)}_n)_{n\geq 0}$ is a branching random walk in $\R_+^2$. 

To conclude the proof, we simply need to consider limits as $\ell\to \infty$. Indeed, we have plainly for every dyadic rational time $t$ and $j\geq 1$ that 
$| b^{(\ell)}_{j,t} - b_{j,t}| \leq 2^{-\ell} $, and it also follows from Lemma~\ref{l:BM2}
that $g^{(\ell)}_{j,t}=g_{j,t}$ provided that $\ell$ is sufficiently large. 
Since the restriction of $\Y_n$ to $[0,t]\times \R_+$ is given by
$$\sum_{j\geq 1}\ind{g_{j,t}=n} \delta_{(b_{j,t}, x_{j,t}(t))},$$
and similarly the restriction of $\Y^{(\ell)}_n$ to $[0,t]\times \R_+$ by
$$\sum_{j\geq 1}\ind{g^{(\ell)}_{j,t}=n} \delta_{(b^{(\ell)}_{j,t}, x_{j,t}(t))},$$
we deduce that $\Y^{(\ell)}_n$ converges vaguely on $\R_+^2$ to $\Y_n$ as $\ell\to \infty$. On the other hand, because children are always born at the right of their parent (and are of course also born after their parent), 
vague convergence suffices to ensure that the branching property for $(\Y^{(\ell)}_n)_{n\geq 0}$ can be transferred to $(\Y_n)_{n\geq 0}$.
\end{proof}

We now turn our attention to the reproduction law of the branching random walk $(\Y_n)_{n\geq 0}$,
and in this direction, it is convenient to handle all together the atoms of the first generation which are born at the same time. That is, we consider the point measure ${\mathbf M}$ on $(0,\infty)\times \calP_+^*$, where
 the atoms $(t,\x)$ of ${\mathbf M}$ are such that 
$\x=(x_k)_{k\geq 1}\in \calP_+^*$ is the ranked sequence of the atoms of the first generation in the branching-stable process which are born at time $t>0$. 
In particular, we have 
$$ \int_{(0,\infty)\times \calP^*_+} \crochet{\x,f(t,\cdot)} {\mathbf M}(\dd t, \dd \x) = \crochet{\Y_1,f},$$
for every measurable nonnegative function $f$ on $(0,\infty)^2$, with the notation $$\crochet{\x,f(t,\cdot)}= \sum_{k\geq 1} f(t,x_k).$$

\begin{lemma}\label{l:scalbrw} In the notation above, ${\mathbf M}$ is a Poisson point process in $(0,\infty)\times \calP_+^*$ with intensity $\dd t \otimes M(\dd \x)$, where $M$ is some measure on $\calP_+^*$ which satisfies the following self-similarity property: for every $c>0$, the image of $M$ by the dilation $\x\mapsto c \x$ is $c^{-\alpha}M$.
\end{lemma}
\begin{proof} Let ${\mathcal A}\subset \calP_+^*$ be a measurable set with $a\coloneqq \sup\{x_1: \x\in {\mathcal A}\}<\infty$, and introduce the process
$$N_{\mathcal A}(t)={\mathbf M}([0,t]\times {\mathcal A}), \qquad t\geq 0,$$
which counts the number of times when the ancestor begets a progeny in ${\mathcal A}$. 
Plainly $N_{\mathcal A}(t)<\infty$ a.s. for every $t\geq 0$, since otherwise the point measure $\S_t$ would have infinitely many atoms in $[0,a]$. The branching property of $(\S_t)_{t\geq 0}$ implies that the counting process 
$N_{\mathcal A}$ has independent and stationary increments in the natural filtration generated by the branching-stable process and its genealogy. Thus $N_{\mathcal A}$ is a Poisson process with intensity which we denote by $M({\mathcal A})$.

Next, if we consider the counting processes $N_{{\mathcal A}_1}, N_{{\mathcal A}_2}, \ldots$ associated to 
pairwise disjoint measurable sets ${\mathcal A}_i\subset \calP_+^*$, we get a family of Poisson processes in the same filtration which have no commun jump-times and therefore they are independent. This entails that ${\mathbf M}$ is a Poisson point process, that $M$ is a measure on $\calP_+^*$ with $M({\mathcal A})<\infty$ whenever $\sup\{x_1: \x\in {\mathcal A}\}<\infty$, and that the intensity of ${\mathbf M}$ is given by $\dd t \otimes M(\dd \x)$.

We finally note that from the self-similarity property of branching-stable processes that 
the reproduction law $\Y_1$ is invariant by the scaling transformation which maps every atom $(t,x)\in(0,\infty)^2$ to
$(c^{-\alpha}t,cx)$. This entails that the law of ${\mathbf M}$ is invariant by the transformation
$(t,\x)\mapsto (c^{-\alpha}t,c\x)$. {\it A fortiori}, the same holds for its intensity $\dd t \otimes M(\dd \x)$, 
and in turn this shows the self-similarity property of $M$. 
\end{proof}

We can now complete the proof of Theorem~\ref{T1}(ii) by checking that the measure 
$M$ in Lemma~\ref{l:scalbrw} has necessarily the form \eqref{defLambda}.

\begin{lemma}\label{lem_triplet}
Let $M$ be a measure on $\calP_+^*$ with $M(\{\x\in \calP_+^*: x_1\leq 1\})<\infty$, such that for every $c>0$, the image of $M$ by the dilation $\x\mapsto c \x$ is $c^{-\alpha}M$.
There exists a unique finite measure $m$ on $\calP_+^*$ with $m(\{\x\in \calP_+^*: x_1\neq 1\})=0$
such that 
$$\int_{\calP_+^*} F(\x) M(\dd \x) = \int_0^{\infty} y^{\alpha-1} \int_{\calP_+^*} F(y\x) m(\dd \x) \dd y ,
$$
for every measurable function $F: \calP_+^*\to \R_+$.
\end{lemma}

\begin{proof}
The self-similarity entails that $M(\{\x\in \calP_+^*: x_1\leq y\}) = c y^{\alpha}$ for all $y>0$, where $c$ is some finite constant. We then consider a regular disintegration of $M$ with respect to $x_1$; specifically there is a kernel $(m(y, \dd \x))_{y>0}$ of finite measures on $\calP_+^*$ such that for every measurable functional $G: \R_+\times \calP^*_+\to \R_+$, 
$$\int_{\calP_+^*}G(x_1, \x)M(\dd \x) = \int_0^{\infty} y^{\alpha-1} \int_{\calP_+^*} G(y, \x) m(y,\dd \x) \dd y.
$$
The self-similarity of $M$ further implies that for every $c>0$, the image of $ m(y,\dd \x)$ by the dilation $\x \mapsto c \x$
coincides with $ m(cy,\dd \x)$ for almost all $y>0$, which enables us to define a finite measure $m$ on $\calP_+^*$
such that for almost all $y>0$, $m(y,\dd \x)$ coincides with the image of $m$ by the dilation $\x\mapsto y\x$, and the representation of the statement follows. 
\end{proof}

\section{Asymptotic behavior of the cumulative distribution function}
\label{sec:martingale}
 The purpose of this section is to analyze the asymptotic behaviors of 
$\S_t([0,x])$ as $x\to\infty$, respectively as $t\to \infty$, where $(\S_t)_{t\geq0}$ is a branching-stable process with scaling exponent $-\alpha$. Of course, the two are closely related, thanks to the scaling property, and with no loss of generality, we shall mainly consider the first and further focus on the case $t=1$.

Recall from the previous section that $(\Z_n)_{n\geq 0}$ is the branching random walk in the quadrant $\R_+^2$ from which the branching-stable process $(\S_t)_{t\geq 0}$ is constructed (thanks to Theorem~\ref{T1}(ii), we may agree from now on that $\S=\X$ in the notation there), and write $\mu_n$ for its intensity measure, that is
$$\int_{(0,\infty)^2} f(t,x)\mu_n(\dd t, \dd x)= \E(\crochet{\Z_n,f}).$$

\begin{lemma}\label{L1} For every $n\geq 1$, there is the identity
\[ \mu_n(\dd t ,\dd x) = \frac{c(\lambda)^n \Gamma(\alpha)^n}{(n-1)! \Gamma(\alpha n)} t^{n-1}x^{n\alpha-1} \dd t \dd x.\]
\end{lemma}

\begin{proof} For $n=1$, this identity has been already checked in \eqref{e:intens1}. Then recall that 
the intensity measures in a branching random walk are given by convolution powers of the intensity of the reproduction law, so $\mu_n=\mu_1^{*n}$. The general formula follows straightforwardly, for instance by computing Laplace transforms.
\end{proof}

For the sake of simplicity, we assume throughout the rest of this work that 
\begin{equation}
 \label{eqn:normalization}
 c(\lambda) \Gamma(\alpha)=1,
\end{equation}
as the general case can be recovered easily by dilation. Hence, Lemma~\ref{L1} reads
\begin{equation}\label{intensL}
 \E(\crochet{\Z_n,f})= \frac{1}{(n-1)! \Gamma(\alpha n)} \int_0^{\infty}
 \int_0^{\infty} f(t,x) t^{n-1}x^{n\alpha-1} \dd t \dd x.
 \end{equation}
This yields the intensity measure of the branching-stable point measure $\S_t$, whose cumulative distribution function bears a simple relation to so-called Wright (generalized Bessel) function
$$\phi(\rho,\beta, z)\coloneqq \sum_{k=0}^{\infty} \frac{z^k}{k! \Gamma(\rho k+\beta)}.$$
We refer to Gorenflo, Luchko and Mainardi \cite{GoLuMa} for a survey of analytic properties of this function and its applications to partitions in combinatorics and to certain PDE's of fractional order. 
\begin{proposition} \label{P1} \begin{enumerate}
\item[(i)] For every $x>0$ and $t>0$, we have
$$\E(\S_t([0,x]))= \phi(\alpha,1,tx^{\alpha}).$$
\item[(ii)] For every $\theta>0$, we have
$$\E\left(\int_{[0,\infty)}\e^{-\theta x} \S_t(\dd x) \right) = \exp(t\theta^{-\alpha}).$$
\end{enumerate}
\end{proposition}

\begin{proof} The first assertion follows immediately from \eqref{intensL}, the construction of $\S_t$ in terms of $(\Z_n)_{n\geq 0}$, and the definition of Wright functions above. The second then derives readily from the first and the identity
$$\int_0^{\infty} \e^{-\theta x} x^{\beta-1}\dd x = \theta^{-\beta} \Gamma(\beta),$$
which concludes the proof.
\end{proof}

We next point at the following consequence of Proposition~\ref{P1}(i), which is obtained by 
specializing Wright's asymptotic expansion (Theorem 2 in Wright \cite{Wright} or Theorem 2.1.2 in \cite{GoLuMa}) of $\phi(\rho,\beta, z)$ as $z\to \infty$ to our setting.
\begin{corollary}\label{C1}
It holds that 
$$\E(\S_1([0,x]))\sim \frac{\exp\left( (\alpha+1)(x/\alpha)^{\alpha/(\alpha+1)}\right)}{\sqrt{2\pi (\alpha+1) \alpha^{1/(\alpha+1)}x^{\alpha/(\alpha+1)}}}\qquad \text{as }x\to \infty.$$
\end{corollary}

The main result of this section is a weak limit theorem for $\S_t([0,x])$ properly normalized, which reinforces considerably Corollary~\ref{C1}. It relies on the following assumption on the measure $\Lambda$ (or, equivalently on $\lambda$): for some $p\in(1,2]$, we have
\begin{equation}\label{eqLLp}
 \int_{\calP} \left(\sum_{k\geq 1}\e^{-x_k}\right)^p \Lambda(\dd \x)=\int_0^{\infty} y^{\alpha-1} \int_{\calP} \left(\sum_{k\geq 1}\e^{-yx_k}\right)^p \lambda(\dd \x)\dd y<\infty.
\end{equation}

\begin{theorem}\label{T2} Assume \eqref{eqLLp} holds for some $p\in(1,2]$. Then as $x\to \infty$, there is the convergence in distribution:
$$
  x^{\frac{\alpha}{2(\alpha+1)}}\exp\left(-(\alpha+1)(x/\alpha)^{\alpha/(\alpha+1)}\right)\S_1([0 , x])
\Longrightarrow \bar W,
$$
where $\bar W$ is a positive random variable in $L^p(\P)$ with 
$$\E(\bar W)=\left(2\pi \alpha^{1/(1+\alpha)}(\alpha+1)\right)^{-1/2}.$$ 
\end{theorem}

The scaling property enables us of course to rephrase Corollary~\ref{C1} in the form 
$$\E(\S_t([0,1]))\sim \frac{\exp\left( (\alpha+1)(t/\alpha^{\alpha})^{1/(\alpha+1)}\right)}{\sqrt{2\pi (\alpha+1) (\alpha t)^{1/(\alpha+1)}}}\qquad \text{as }t\to \infty,$$
and, in turn, provided that \eqref{eqLLp} holds, Theorem~\ref{T2} as
\begin{equation}\label{eqT2bis}
  t^{\frac{1}{2(\alpha+1)}}\exp\left(-(\alpha+1) (t/\alpha^{\alpha})^{1/(\alpha+1)}\right)\S_t([0 , 1])
\Longrightarrow \bar W\qquad \text{as } t\to \infty.
\end{equation}
Still assuming \eqref{eqLLp}, it is interesting to point out that, for every $a>0$, there is also the {\em strong} (i.e. almost sure) convergence
\begin{equation}\label{eqT2ter}
  \lim_{t\to \infty}\sqrt t 
  \exp\left(-(\alpha+1)(a/\alpha)^{\alpha/(\alpha+1)}t\right)\S_t([0 , at])
= a^{-\frac{\alpha}{2(\alpha+1)}} \bar W(a),
\end{equation}
where $\bar W(a)$ is a random variable with the same distribution as $\bar W$; see the forthcoming Corollary~\ref{C2} for an (essentially) equivalent result. Formally, one recovers \eqref{eqT2bis} from \eqref{eqT2ter} by taking $a=1/t$. We stress however that the convergence in \eqref{eqT2ter} holds in the strong sense, whereas those in Theorem~\ref{T2} and \eqref{eqT2bis} hold only in the weak sense; see Remark~\ref{r:rcl} below.

The proof of Theorem~\ref{T2} relies crucially on results due to Biggins \cite{Biggins77,Biggins92}. We observe from the branching property and Proposition~\ref{P1}(ii), that for every $\theta>0$, the process
$$W_t(\theta)\coloneqq \exp(-t\theta^{-\alpha})\int_{[0,\infty)}\e^{-\theta x}\S_t(\dd x)\,,\quad t\geq 0$$
is a positive martingale, with terminal values
$W(\theta)\coloneqq \lim_{t\to \infty} W_t(\theta).$
\begin{proposition}\label{P2}
 \begin{itemize}
\item[(i)] The process $(W(\e^{r}))_{r\in \R}$ is stationary.

\item[(ii)]  If \eqref{eqLLp} holds for some $p\in(1,2]$, then for every $\theta >0$, the martingale $(W_t(\theta))_{t\geq 0}$ is bounded in $L^p(\P)$ with terminal value $W(\theta)>0$ a.s.
\end{itemize}
\end{proposition}

\begin{proof} (i) Recall from our notation for dilations that for every point measure $\x\in\calP$ and $a>0$, there are the identities
$$\crochet{a\x, g}=\int_{[0,\infty)}g(y) (a\x)(\dd y)=\sum_{k\geq 1} g(ax_k)=
\int_{[0,\infty)}g(ay) \x(\dd y).$$
The scaling property of Theorem~\ref{T1}(i) implies that for every $c>0$, the processes
$\left(c\theta \S_{t}: t\geq0, \theta>0\right)$ and $\left(\theta \S_{c^{-\alpha}t}:t\geq0, \theta>0\right)$ have the same distribution. Hence the limits
$$ \lim_{t\to \infty} \e^{-t (c \theta)^{-\alpha}} \int_{[0,\infty}\e^{-c\theta x}\S_t(\dd x)= W(c\theta)$$
and 
$$
\lim_{t\to \infty} \e^{-t c^{-\alpha} \theta^{-\alpha}} \int_{[0,\infty}\e^{-\theta x}\S_{c^{-\alpha}t}(\dd x)= W(\theta)$$
have the same law, simultaneously for all $\theta>0$. 

(ii) Essentially, this follows from the version of Biggins' martingale convergence theorem for branching L\'evy processes
which has been recently established in \cite{BeMa2}. To fit the setting of \cite{BeMa2}, we rather consider the reflected process $(-\S_t)_{t\geq 0}$. 

Recall from the discussion after the proof of 
Theorem~\ref{T1}(i) that $(-\S_t)_{t\geq 0}$ is a branching L\'evy process with L\'evy measure $\Lambda$. Further Proposition~\ref{P1}(ii) identifies the so-called cumulant as $\kappa(\theta)=\theta^{-\alpha}$ (this can also be checked directly from the L\'evy-Khintchine type formula (5.4) in \cite{BeMa}). The conditions $ \kappa(p\theta)<p \kappa(\theta)$ 
and $\kappa(q\theta)<\infty$ for some $q>p$ obviously hold for all $\theta >0$ and $p>1$. We then deduce from Proposition 1.4 in \cite{BeMa2} that the martingale $(W_t(\theta))_{t\geq 0}$ is bounded in $L^p(\P)$ whenever 
$$ \int_{\cal P} \crochet {\tilde \x, \e^{\theta\cdot}}^p\ind{\crochet {\tilde \x, \e^{\theta\cdot}}>2} \Lambda(\dd \x)<\infty,$$
with the notation $\tilde \x=(0,-x_1, -x_2, \ldots)$,
that is 
$$ \int_{\cal P} \left( \crochet {\x, \e^{-\theta\cdot}}+1\right)^p \ind{ \crochet {\x, \e^{-\theta\cdot}}>1} \Lambda(\dd \x) <\infty.$$
It is easily seen that this is equivalent to \eqref{eqLLp}.
\end{proof}
\begin{remark}
Using the branching property of the branching Lévy process, one can obtain an identity in distribution that characterizes the law of the terminal value $W(\theta)$. This equation is often called the smoothing transform (see \cite{Liu98,BiK05,ABM12}). The process $(W(\e^r), r \in \R)$ therefore as multiple features of interest, that justify further investigations.
\end{remark} 

We can now lift from Corollary 4 of Biggins \cite{Biggins92} the following local limit theorem.
\begin{corollary}\label{C2} Assume \eqref{eqLLp} for some $p\in(1,2]$, and let $f:\R\to \R$ be a directly Riemann integrable function with compact support. Then 
\begin{eqnarray*}
& &\lim_{t\to \infty} \sqrt t \exp\left(-(\alpha+1)\theta^{-\alpha}t\right)\int_{\R_+}f(y-\alpha \theta^{-\alpha-1}t)\S_t(\dd y)\\
&=& \frac{W(\theta)}{\sqrt{2\pi \alpha(\alpha+1)\theta^{-\alpha-2}}} \int_{\R} f(y) \e^{\theta y}\dd y,\end{eqnarray*}
where the limit is uniform for $\theta$ in compact subsets of $(0,\infty)$, almost surely. 
\end{corollary}
Corollary~\ref{C2} describes the precise asymptotic behaviour of the point measure $\S_t$ in the neighbourhood of $-\kappa'(\theta )t$ for any $0<\theta<\infty$, where $\kappa'(\theta )=-\alpha \theta^{-\alpha-1}$, and in this regime, the limit involves the terminal value $W(\theta)$ of the additive martingale as an asymptotic weight. Roughly speaking, recalling from Proposition~\ref{P1} that the distribution of $W(\theta)$ does not depend on $\theta$ and using the scaling property of $(\S_t)_{t\geq0}$, yields the proof of Theorem~\ref{T2} by 
establishing that a similar behaviour holds for the boundary case $\theta = 0$, but with weak convergence instead of strong convergence.

\begin{proof}[Proof of Theorem~\ref{T2}] Specialising Corollary~\ref{C2} for $\theta=1$ and the directly Riemann integrable $f(x)=\ind{-a\leq x\leq0}$ for some given $a>0$, we get
$$
\lim_{t\to \infty} {\sqrt{2\pi t \alpha(\alpha+1)}} \e^{-(\alpha+1)t}\S_t([\alpha t -a, \alpha t])
= (1-\e^{-a}) W(1)\quad \text{a.s.}
$$
By the scaling property, this translates into the weak convergence as $t\to \infty$
$$
 {\sqrt{2\pi t \alpha(\alpha+1)}} \e^{-(\alpha+1)t}\S_1([\alpha t^{1+1/\alpha} -at^{1/\alpha} , \alpha t^{1+1/\alpha} ])
\Longrightarrow (1-\e^{-a}) W,
$$
where $W$ denotes a random variable distributed as $W(1)$. 
Since $a$ can be picked arbitrarily large, we deduce that for all $y>0$, one has
$$\liminf_{t\to \infty} \P\left({\sqrt{2\pi t \alpha(\alpha+1)}} \e^{-(\alpha+1)t}\S_1([0 , \alpha t^{1+1/\alpha} ]) >y \right)\geq \P\left(W >y\right).$$

On the one hand, 
\begin{eqnarray*} 
&&\int_0^{\infty} \P\left({\sqrt{2\pi t \alpha(\alpha+1)}} \e^{-(\alpha+1)t}\S_1([0 , \alpha t^{1+1/\alpha} ]) >y \right)\dd y \\
&=& 
 {\sqrt{2\pi t \alpha(\alpha+1)}}\e^{-(\alpha+1)t}\E\left(\S_1([0,\alpha t^{1+1/\alpha} ])\right),
\end{eqnarray*}
and we see from Corollary~\ref{C1} that this quantity converges to $1$ as $t\to \infty$.
On the other hand, recall from Proposition~\ref{P2} that 
$$\int_0^{\infty} \P\left(W >y\right) \dd y = \E(W) = 1.$$
We deduce from a version of Scheff\'e's lemma that actually
$$\lim_{t\to \infty} \P\left({\sqrt{2\pi t \alpha(\alpha+1)}} \e^{-(\alpha+1)t}\S_1([0 , \alpha t^{1+1/\alpha} ]) >y \right)\\
= \P\left(W >y\right) \ \hbox{in } L^1(\dd y).
$$
Setting $x=\alpha t^{1+1/\alpha}$ and 
$\bar W = {W}/{\sqrt{2\pi (\alpha+1)\alpha^{1/(\alpha+1)}}},$
yields our claim. 
\end{proof}

\begin{remark}\label{r:rcl}
Note from the stationarity of $(W(\e^r), r \in \R)$, that one has plainly the weak convergence 
\[
 W(\theta) \Longrightarrow W \quad \text{both as $\theta \to 0+$ and $\theta\to \infty$.}
\]
To see that the convergence above cannot hold in the almost sure sense, observe from stationarity that for any $r\in\R$, the pair $(W(1),W(2))$ has the same distribution as $(W(\e^r),W(2\e^r))$. Thus, if $W(\e^r)$ converged almost surely as $r\to \infty$, say towards $W(\infty)$, then so would $W(2\e^r)$, which would force $W(1)=W(2)$ a.s. This is absurd, because 
obviously for every $t>0$, $W_t(1)\neq W_t(2)$ with positive probability, so the terminal values of these two uniformly integrable martingales cannot coincide a.s. This explains why the convergence in Theorem~\ref{T2} cannot be strengthened to a.s. convergence either.
\end{remark}

\section{On the genealogy of atoms}
\label{sec:generation}

The aim of this section is to investigate the genealogy of atoms of a branching-stable point measure with negative scaling exponent, viewed as individual in a CMJ  branching process. We first describe the ancestral lineage of a typical particle at generation $n$ (often referred to as the \emph{spine}), and then turn our attention to asymptotic results concerning the minimal position of such particles.  

Recall first that $\S_1$ is constructed as the superposition of the projections on the $x$-axis of the point measures $\ind{t\leq 1}\Z_n(\dd t, \dd x)$ for $n\geq 0$. In this direction, if $({t},x)$ is an atom in $\Z_n$ with ${t}\leq 1$, we call ${t}$ the birth-time of the atom $x$ in $\S_1$ and $n$ the generation of $x$.
 In this case, we further denote by
\[
0 < t_1 < \ldots < t_n : = {t},
\quad
\text{and}
\quad
0 < x_1 < \ldots < x_n : = x,
\] 
the birth-times and, respectively, the positions of the ancestors of $x$. To shorten the notation we shall often use an upper-bar, {\it e.g.} $\bar{t}$ and $\bar{x}$, to denote the corresponding $n$-dimensional vectors. 

The first result in this Section is a many-to-one type formula which describes the distribution of $(\bar{t}, \bar{x})$ for ``typical'' atoms at generation $n$. In this direction, let $(S_n)_{n\geq 0}$ denote a random walk on $\R_+$ whose increments have the
gamma distribution with parameter $(\alpha, 1)$, that is with density $\Gamma(\alpha)^{-1} x^{\alpha-1}\e^{-x}$ for $x>0$. Let also $(U_i)_{i\geq 1}$ be i.i.d. copies of a uniform random variable on $[0,1]$, and set
$A_n := U_1 \cdots U_n$ for $n\geq 1$, where the notation $A$ refers to age.

\begin{proposition}
\label{many-to-one} For every $n\geq 1$ and  $f : \big( [0,1] \times \R_+ \big)^n \to \R_+ $ measurable, we have:
\begin{equation}\label{eq:many-to-one}
\E \left( \int_{[0,1]\times \R_+} f \big( (1-t_i, x_i)_{i\leq n}  \big) \Z_n( \dd t,\dd x) \right)
= \E \Big(  f \big( (A_i, S_i )_{i\leq n}  \big) \e^{ S_n}  \textstyle \prod\limits_{i=1}^{n-1} A_i \Big).
\end{equation}
\end{proposition}
\begin{remark} As a check, a direct application of Proposition~\ref{many-to-one} with $f (\bar t, \bar x) =\ind{x_n \leq a}$ gives 
\begin{align*}
\E \left( \Z_n \big( [0,1]\times [0,a] \big) \right) & = \E \left( \int_{[0,1]\times \R_+} \ind{x_n \leq a} \Z_n(\dd {t}, \dd x)\right) \\
& = \E \Big(\ind{S_n \leq a} \e^{ S_n} \textstyle \prod\limits_{i =1}^{n-1} A_i \Big) \\
& = \frac{1}{(n-1)!} \int_{0}^{a} \e^{ y} \frac{y^{\alpha n-1} \e^{- y}}{\Gamma(\alpha n)} \dd y\\
& = \frac{a^{\alpha n}}{n! \Gamma \big(\alpha n+1 \big)}, 
\end{align*}
which is in agreement with Lemma~\ref{L1}. 
\end{remark}

\begin{proof} The proof is done by induction. The case $n=1$ merely follows from \eqref{e:intens1}. 
Next, assume that the formula has been established for some $n\geq 1$, consider a measurable function $f : \big( [0,1] \times \R_+ \big)^{n+1} \to \R_+ $ and define for every $0< t_1<\ldots < t_n<1$ and $0<x_1< \ldots < x_n$
\begin{eqnarray*}& &F \big( (1-t_i, x_i)_{i\leq n}  \big) \\
&=& \E\left( \int_{[0,1-t_n]\times \R_+} \!\!\!\!\!\!\!\! \!\!\!\!\!\!\!\! \!\!\!\!\!\!\!\! f \big( (1-t_1,x_1), \ldots, (1-t_n, x_n), (1-t_n-t, x_n+y) \big) \Z_{1}(\dd t, \dd y)\right).
\end{eqnarray*}
The branching property and (conditional) Fubini's theorem show that
\begin{equation}\label{eq_123212321}
\begin{split}
& \E \left( \int_{[0,1]\times \R_+} f \big( (1-t_j, x_j)_{j\leq n+1}  \big) \Z_{n+1}(\dd t, \dd x) \right) \\
& = \E \left( \int_{[0,1]\times \R_+} F \big( (1-t_i, x_i)_{i\leq n}  \big)  \Z_{n}(\dd t, \dd x) \right)\\
&= \E \left( F \big( (A_i, S_i)_{i\leq n} \big) \e^{ S_n} \textstyle \prod\limits_{i=1}^{n-1} A_i  \right),
 \end{split}
\end{equation}
where the last equality stems from the induction assumption.

By Lemma~\ref{L1} and Campbell's formula, we see that $F \big( (1-t_i, x_i)_{i\leq n}  \big)$ is equal to 
\[
\int_{0}^{1-t_n} \int_0^{\infty} f \big( (1-t_1,x_1), \ldots, (1-t_n, x_n), (1-t_n-s, x_n+y) \big) y^{-\alpha-1} c(\lambda) \dd s \dd y.
\]
Next, we use the change of variables $u = s/(1-t_n)$ and the normalization assumption \eqref{eqn:normalization} to rewrite the above as
\[
(1-t_n) \int_{0}^{1} \int_0^{\infty} f \big( (1-t_1,x_1), \ldots, \big( (1-t_n)(1-u), x_n+y) \big) \e^{ y } \frac{y^{-\alpha-1} \e^{- y } }{\Gamma(\alpha)}  \dd u \dd y. 
\]
Let $U_{n+1}$ be a uniform random variable and $\gamma_{n+1}$ an independent gamma$(\alpha,1)$ random variable, so the above, and hence $F \big( (1-t_i, x_i)_{i\leq n}  \big) $, reads
\begin{equation}\label{eq_11223344556}
 (1-t_n) \E \left( f \big( (1-t_1,x_1), \ldots, \big( (1-t_n)U_{n+1}, x_n+\gamma_{n+1}) \big) \e^{ \gamma_{n+1} } \right).
\end{equation}

Plugging \eqref{eq_11223344556} into \eqref{eq_123212321} yields
\begin{eqnarray*}
& &\E \left( \int_{[0,1]\times \R_+} f \big( (1-t_j, x_j)_{j\leq n+1}  \big) \Z_{n+1}(\dd t, \dd x) \right) \\
&=& \E \left( f \big( (A_i, S_i)_{i\leq n+1} \big) \e^{ S_{n+1}} \textstyle \prod\limits_{i=1}^{n} A_i  \right),
\end{eqnarray*}
which concludes the proof.
\end{proof}

We borrow the terminology from the branching random walk literature and call $(A_n, S_n)$ the \emph{spine} of $\S_1$ in generation $n$. The many-to-one formula is a powerful tool in the study of branching random walks. Classically, it is used to calculate quantities depending on the first moment of $\Z_n$. For example, one can recover the first formula in Proposition~\ref{P1} with $t=1$ by choosing $f(\bar t, \bar x)= \ind{x_n \leq x}$ in \eqref{eq:many-to-one} and summing over all possible $n \in \N$. 

We shall now rely on Lemma~\ref{many-to-one} to derive asymptotic results about the minimal particle in $\Z_n$ restricted to the strip $[0,1] \times \R_+$, namely,
\begin{equation}\label{eq_minimum_BRW_Strip}
\mathfrak{z}_n:= \sup\{ z; \Z_n \big( [0,1]\times [0,z] \big) = 0  \}. 
\end{equation}
By scaling, $\mathfrak{z}_n$ determines the law of the minimal position in $\Z_n$ restricted to any strip $[0,a] \times \R_+$ for all $a >0$. 

For a classical branching random walk, the asymptotic behavior of the minimal position $\mathfrak{m}_n$ is now well-known. Addario-Berry and Reed \cite{ABR09} and Hu and Shi \cite{HuS09} independently proved that $\mathfrak{m}_n = n v - c \log n + O_\P(1)$ as $n \to \infty$, where $v \in \R$ and $c>0$ are explicit constants, and $O_\P(1)$ represents a tight sequence of random variables. This result was improved by Aïdékon \cite{Aid13}, who proved that
\[
 \lim_{n \to \infty} \mathfrak{m}_n - nv + c \log n = G \quad \text{ in law,}
\]
where $G$ is a mixture of Gumbel random variables.

In our framework, the behavior of $\mathfrak{z}_n$ is rather different: 
\begin{proposition}\label{prop:asympt_position_min}
It holds that
\begin{equation}
 \lim_{n \to \infty} \frac{\mathfrak{z}_n}{n^{(\alpha + 1)/\alpha}} = c_\alpha \quad \text{ in probability,}
\end{equation}
with $c_\alpha := \alpha \e^{-(\alpha+1)/\alpha}$
\end{proposition}

\begin{proof}[Proof of the lower-bound]
By first moment estimates, the many-to-one Proposition~\ref{many-to-one} yields for every $\eta>0$
\begin{align}\label{lowe_bound_minmal_pos}
\P\big(\mathfrak{z}_n \leq (1-\eta) c_\alpha n^{\frac{\alpha+1}{\alpha}} \big) &
= \P\big( \Z_n \big( [0,1] \times [0, (1-\eta) c_\alpha n^{\frac{\alpha+1}{\alpha}}]) \geq 1 \big) \notag
\\
& \leq \frac{ \big[(1-\eta) c_\alpha \big]^{\alpha n} n^{(\alpha+1) n }}{n! \Gamma \big(\alpha n+1 \big)} \notag
\\
&\lesssim \exp\Big( n \alpha \log(1-\eta) \Big), \notag
\end{align}
where the notation $\lesssim $ means that the left-hand side is smaller than the right-hand side up to a constant (not depending on $n$) factor. In particular, this last quantity is summable over $n$, so in view of Borel-Cantelli lemma, we get
\[
\liminf_{n \to \infty} \mathfrak{z}_n n^{-\frac{\alpha+1}{\alpha}} \geq c_\alpha
\quad \text{a.s.}
\]
\end{proof}

The proof of the upper-bound relies on shape-type results for the convex hull of $\Z_n$ due to Biggins \cite{Big78}. First, we use \eqref{e:intens1} to compute the log-Laplace transform $\kappa(a,b)$ of $Z_1$ for $a,b>0$:
\[
 \kappa(a,b) \coloneqq \log \E\left( \int \e^{-(a t + b x)} \Z_1(\dd t, \dd x) \right) 
= -\log (a b^\alpha).
\]
We then compute its Legendre transform $\kappa^*$:
\begin{eqnarray*}
 \kappa^*(p,q) &\coloneqq& \inf_{a,b>0} \big\{ \kappa(a,b) + p a + b q \} \\
 &=& \left\{\begin{matrix}  1 + \alpha + \log p + \alpha \log (q/\alpha) &\text{if }p,q>0,\\
\infty & \text{otherwise.}\end{matrix} \right.
\end{eqnarray*}

 Next we let $\mathcal{J}_n$ be the set of atoms in $\Z_n$ re-scaled by $1/n$, so that
\[
\textstyle \Z_n := \sum_{z \in \mathcal{J}_n} \delta_{n z},
\] 
and denote by $\mathcal{H}_n$ its convex hull. We readily lift from \cite[Section~3]{Big78} the following statement.

\begin{proposition} \label{prop_shape} Let $C_\alpha := \{ (p,q) : \kappa^*(p,q) \geq 0 \} $ and $\mathrm{int} \, C_\alpha$ be its interior. Then we have
\[
\mathrm{int} \, C_\alpha \subset \liminf_{n \to \infty} \mathcal{H}_n \subset \limsup_{n \to \infty} \mathcal{H}_n \subset C_\alpha
\quad \text{a.s.}
\]
\end{proposition}

With Proposition~\ref{prop_shape} in hands, we can complete the proof of Proposition~\ref{prop:asympt_position_min} by establishing the converse upper-bound. We stress that our argument only yields an upper-bound in probability, because we rely on the scaling property of $\Z_n$ to apply Proposition~\ref{prop_shape}.

\begin{proof}[Proof of the upper-bound]
To start with, notice that the set $C_\alpha$ in Proposition~\ref{prop_shape} is given by
\[
C_\alpha = \{ (p,q) : pq^\alpha \geq \alpha^\alpha \e^{-(1+\alpha)} \},
\]
and recall that $c_\alpha = \alpha \e^{-(\alpha+1)/\alpha}$. 

Next, the scaling property of $\Z_n$ yields that $\Z_n [0,1]\times [0,t^{1/\alpha} x]$ and $\Z_n [0,t]\times [0,x]$ have the same distribution. Therefore, fixing $\eta>0$ arbitrarily small and setting
$x = (c_\alpha+\eta) n $ and $t = n$, we get
\[
\P\big( \Z_n( [0,1]\times [0,(c_\alpha+\eta) n^{(\alpha+1)/\alpha} ]) = 0 \big) =
\P\big( \Z_n ([0,n]\times [0,(c_\alpha+\eta) n ] )= 0 \big).
\]
Recalling that $\mathcal{H}_n$ stands for the convex hull obtained from the atoms of $\Z_n$ re-scaled by $1/n$, the above yields
\begin{equation}\label{eq:leftmost_particle_convex_hull}
\P\left( \mathfrak{z}_n n^{-\frac{\alpha+1}{\alpha}} \geq c_\alpha+\eta \right) \leq \P\big( (1, c_\alpha +\eta) \notin \mathcal{H}_n \big)
\end{equation}
Then, observe that $\kappa^\ast(1, c_\alpha+ \eta) > 0 $ and hence $(1, c_\alpha+\eta ) \in \mathrm{int} \, C_\alpha$. 
On the other hand, Proposition~\ref{prop_shape} states that 
$\mathrm{int} \, C_\alpha \subset \liminf \mathcal{H}_n$ a.s. Putting the pieces together, this implies that the right-hand side of \eqref{eq:leftmost_particle_convex_hull} tends to zero as $n\to \infty$, concluding the proof. 
\end{proof}

\paragraph*{Acknowledgments:} 
The work of J.B. and A.C. is supported by the Swiss National Science Foundation 200021\underline{{ }{ }}163170. B.M. is partially supported by the ANR MALIN project (ANR-16-CE93-0003).

\bibliographystyle{alpha}
\newcommand{\etalchar}[1]{$^{#1}$}

\end{document}